\documentclass[english]{article}
\usepackage[T1]{fontenc}
\usepackage[latin9]{inputenc}
\usepackage{geometry}
\usepackage{mathrsfs}
\usepackage{amsmath}
\usepackage{amssymb}
\usepackage{lmodern}

\makeatletter

\title{Energy-Based In-Domain Control of a Piezo-Actuated Euler-Bernoulli Beam\thanks{This work has been supported by the Austrian Science Fund (FWF) under
grant number P 29964-N32.}} 

\author{T. Malzer\thanks{Institute of Automatic Control and Control Systems Technology,
Johannes Kepler University Linz, Altenbergerstrasse 66, 4040 Linz,
Austria (e-mail: \{tobias.malzer\_1, markus.schoeberl\}@jku.at)}, H. Rams\thanks{B\&R Industrial Automation GmbH, B\&R Stra\ss{}e 1, 5142 Eggelsberg, Austria (e-mail: hubert.rams@br-automation.com).}, M. Sch\"{o}berl
\footnotemark[2]
}

\date{}

\usepackage{pgfplots}
 \pgfplotsset{compat=newest}
  \usetikzlibrary{plotmarks}
  \usepackage{grffile}

  \usepackage{amsthm}
  \usepackage{microtype}

  \usepackage{amsmath}
  \usepackage{mathrsfs}
  \usepackage{subcaption}
  \usepackage{graphics,xcolor}
\definecolor{P285U}{cmyk}{0.89,0.43,0.0,0.0}
\definecolor{P285U_font}{cmyk}{0.89,0.43,0.0,0.4}
\definecolor{lgray}{cmyk}{0,0,0,0.2}
\definecolor{myblue}{cmyk}{100,75,0,0}
\definecolor{jkuBlue}{RGB}{4,110,152}
\definecolor{jkuBlue}{RGB}{0,120,170}
\definecolor{jkuCyan}{RGB}{100,180,190}
\definecolor{jkuYellow}{RGB}{230,195,35}
\definecolor{jkuGrey}{RGB}{125,130,140}
\definecolor{jkuDarkGrey}{RGB}{51,51,51}
\definecolor{jkuLightGreen}{RGB}{195,215,75}
\definecolor{jkuGreen}{RGB}{115,180,85}
\definecolor{jkuPurple}{RGB}{145,75,130}
\definecolor{jkuRed}{RGB}{205,90,80}

\usepackage{mdframed}

\newtheorem{thm}{Theorem}

\newtheorem{prop}[thm]{Proposition}

\theoremstyle{definition}

\newtheorem{exmp}{Example}
\newtheorem{rem}{Remark}

\makeatother

\usepackage{babel}
\begin{document}
\maketitle \thispagestyle{empty} \pagestyle{empty}

\begin{abstract}
The main contribution of this paper is the extension of the well-known boundary-control
strategy based on structural invariants to the control of infinite-dimensional systems
with in-domain actuation. The systems under consideration, governed by partial differential
equations, are described in a port-Hamiltonian setting making heavy use of the underlying
jet-bundle structure, where we restrict ourselves to systems with 1-dimensional spatial domain
and 2nd-order Hamiltonian. To show the applicability of the proposed approach, we develop a
dynamic controller for an Euler-Bernoulli beam actuated with a pair of piezoelectric patches
and conclude the article with simulation results.
\end{abstract}

\section{Introduction}

For the description of finite-dimensional systems, due to the illustration
of the underlying physical effects, the port-Hamiltonian (pH) system
representation has proven to be an adequate framework. A major advantage
is the close relation between the total energy of the system and the
corresponding evolution equations, providing an insight into the energy
flows within the system and with the system environment. In particular
the fact that so-called energy-ports appear, makes the pH-system representation
interesting for control-engineering applications, see \cite{Schaft2000,Ortega2001}
for instance. The basic idea of such appropriate control strategies
is to design a closed-loop system exhibiting a desired behaviour.
Therefore, the objective is to shape the total energy of the system
and increase the dissipation rate, or even modify the structure of
the system dynamics.

Some of these well-known control strategies, like the so-called energy-Casimir
method, have already been extended to systems governed by partial
differential equations (PDEs), where for the most parts the occurrence
of boundary-energy ports is exploited. It must be mentioned that the
pH-system representation in the infinite-dimensional scenario is not
unique, see \cite{Schoeberl2013b} for a comparison of the Stokes-Dirac
approach and an approach based on jet-bundle structures by means of
a mechanical example. As a consequence, the generation of the boundary
ports strongly depends on the chosen approach. In \cite{Schaft2002,Macchelli2004a,Gorrec2005},
the boundary-energy flow for infinite-dimensional systems formulated
within the Stokes-Dirac scenario is investigated, while in \cite{Ennsbrunner2005,Schoeberl2014a,Malzer2018}
the generation of boundary ports for the jet-bundle approach is discussed.
With regard to control by interconnection based on structural invariants
(Casimir functionals), these boundary ports shall be used to couple
the infinite-dimensional plant to a dynamic controller. This control
strategy provides the opportunity to inject additional damping into
the closed loop, and moreover, the mentioned Casimir functionals are
used to relate (some) of the controller states to the plant in order
to shape the energy of the system, see, e.g., \cite{Macchelli2004b,Macchelli2004}
for the controller design in the Stokes-Dirac framework and \cite{Siuka2011a,Schoeberl2013a,Rams2017a}
for the jet-bundle approach. From a mathematical point of view, an
essential feature is that a system governed by PDEs is coupled at
(a part of) the boundary with a system that is described by ordinary
differential equations (ODEs).

It is worth stressing that the application of boundary-control schemes
constitutes a restriction with regard to the energy flows. Furthermore,
there are many actuators, like, e.g., piezoelectric actuators, which
do not operate on the boundary but within the spatial domain. In view
of these facts, it seems natural to extend the known boundary-control
schemes to the control of infinite-dimensional systems with in-domain
actuation. There, a distinction needs to be drawn. On the one hand,
we consider systems with lumped inputs that may act on a part of the
spatial domain. For this system class, the interconnection of the
system with a dynamic controller corresponds to a coupling of a PDE
and an ODE-system within the spatial domain of the infinite-dimensional
system. On the other hand, the scenario of systems with distributed
input densities would call for infinite-dimensional controllers, which
represents the coupling of a PDE with a PDE. For stability investigations
of infinite-dimensional systems, usually functional-analytic methods
are used, see e.g. \cite{Jacob2012}. This framework also allows to
investigate the well-posedness of a problem. However, in this contribution
we assume well-posedness and confine ourselves on energy considerations.
Thus, no detailed stability investigations are carried out.

To demonstrate the proposed control strategy, we consider pH-systems
with 2nd-order Hamiltonian on 1-dimensional spatial domains. For such
systems, in \cite{Rams2017a} a dynamic boundary controller has been
derived. Now, we intend to develop an in-domain control strategy for
this system class. Therefore, the main contributions of this paper
are as follows. First, we state a proper pH-system representation
of a piezo-actuated Euler-Bernoulli beam in Section 3. Furthermore,
we derive an appropriate control methodology for infinite-dimensional
pH-systems with lumped inputs that may act on a part of the spatial
domain and show the capability of the approach by means of simulation
results for the piezo-actuated Euler-Bernoulli beam, see Section 4.

\section{Notation and Preliminaries}

Throughout this paper, we make heavy use of differential-geometric
methods, where the notation is similar to that of \cite{Giachetta1997}.
Formulas are kept short and readable by applying tensor notation and
using Einsteins convention on sums. However, the ranges of the used
indices are not indicated when they are clear from the context. We
use the standard symbols $\wedge$, $\rfloor$ and $\mathrm{d}$ denoting
the exterior (wedge) product, the natural contraction between tensor
fields and the exterior derivative, respectively. To avoid exaggerated
notation, the use of pull-back bundles is omitted. Furthermore, the
expression $C^{\infty}(\mathcal{M})$ denotes the set of all smooth
functions on a manifold $\mathcal{M}$.

In this contribution, we investigate systems governed by PDEs in a
pH-setting. Therefore, we introduce some geometrical structures and
begin with defining a so-called bundle $\pi:\mathcal{E}\rightarrow\mathcal{B}$,
which allows a clear distinction between dependent and independent
coordinates. Here, the base manifold $\mathcal{B}$ is equipped with
the independent coordinate $z^{1}$ as we confine ourselves to 1-dimensional
spatial domains. Consequently, the boundary $\partial\mathcal{B}$
is zero-dimensional and the restriction of a mathematical expression
to $\partial\mathcal{B}$ is indicated with $(\cdot)|_{\partial\mathcal{B}}$.
As the total manifold $\mathcal{E}$ comprises the dependent coordinates
$x^{\alpha}$, with $\alpha=1,\ldots,n$, as well, it is equipped
with $(z^{1},x^{\alpha})$. Moreover, $\pi$ is a surjective submersion
from the total manifold $\mathcal{E}$ to the base manifold $\mathcal{B}$
and is called projection. Next, we consider (higher-order) jet manifolds
to be able to introduce derivative coordinates (jet variables). For
instance, the 4th jet manifold $\mathcal{J}^{4}(\mathcal{E})$ possesses
the coordinates $(z^{1},x^{\alpha},x_{1}^{\alpha},x_{11}^{\alpha},x_{111}^{\alpha},x_{1111}^{\alpha})$,
where, exemplarily, $x_{11}^{\alpha}$ denotes the 2nd-order derivative
coordinate, i.e. the 2nd derivative of $x^{\alpha}$ with respect
to the independent coordinate $z^{1}$. 

Furthermore, we introduce the so-called tangent bundle $\tau_{\mathcal{E}}:\mathcal{T}(\mathcal{E})\rightarrow\mathcal{E}$
equipped with the coordinates $(z^{1},x^{\alpha},\dot{z}^{1},\dot{x}^{\alpha})$,
where the abbreviations $\partial_{1}=\partial/\partial z^{1}$ and
$\partial_{\alpha}=\partial/\partial x^{\alpha}$ denote the fibre
bases of the bundle. An important subbundle of $\tau_{\mathcal{E}}$
is the vertical tangent bundle $\nu_{\mathcal{E}}:\mathcal{V}(\mathcal{E})\rightarrow\mathcal{E}$,
which possesses the coordinates $(z^{1},x^{\alpha},\dot{x}^{\alpha})$.
Since the relation $\dot{z}^{1}=0$ holds, the vertical vector field
$v=v^{\alpha}\partial_{\alpha}$ is tangent to the fibres of $\mathcal{E}$.
Furthermore, the 2nd prolongation of a vertical vector field $v$
is given by $j^{2}\left(v\right)=v^{\alpha}\partial_{\alpha}+d_{1}(v^{\alpha})\partial_{\alpha}^{1}+d_{1}(d_{1}(v^{\alpha}))\partial_{\alpha}^{11}$
and makes use of the total derivative $d_{1}=\partial_{1}+x_{1}^{\alpha}\partial_{\alpha}+x_{11}^{\alpha}\partial_{\alpha}^{1}+x_{111}^{\alpha}\partial_{\alpha}^{11}+x_{1111}^{\alpha}\partial_{\alpha}^{111}+\ldots$
together with the abbreviations $\partial_{\alpha}^{1}=\partial/\partial x_{1}^{\alpha}$
and $\partial_{\alpha}^{1\ldots1}=\partial/\partial x_{1\ldots1}^{\alpha}$.

Further important bundles are the cotangent bundles $\tau_{\mathcal{B}}^{*}=\mathcal{T}^{*}\left(\mathcal{B}\right)\rightarrow\mathcal{B}$
and $\tau_{\mathcal{E}}^{*}=\mathcal{T}^{*}\left(\mathcal{E}\right)\rightarrow\mathcal{E}$
possessing the coordinates $(z^{1},\dot{z}_{1})$ and $(z^{1},x^{\alpha},\dot{z}_{1},\dot{x}_{\alpha})$,
respectively, where the holonomic bases are denoted by $\mathrm{d}z^{1},\mathrm{d}x^{\alpha}$.
These bundles allow to locally define one-forms according to $\varpi=\varpi_{1}\mathrm{d}z^{1}$
and $\omega=\omega_{1}\mathrm{d}z^{1}+\omega_{\alpha}\mathrm{d}x^{\alpha}$,
with $\varpi_{1}\in C^{\infty}(\mathcal{B})$ and $\omega_{1},\omega_{\alpha}\in C^{\infty}(\mathcal{E})$.
In what follows, we are interested in one-forms with coefficients
depending on derivative variables. More precisely, we focus on (Hamiltonian)
densities $\mathfrak{H}=\mathcal{H}\Omega$ with $\mathcal{H}\in C^{\infty}(\mathcal{J}^{2}(\mathcal{E}))$,
i.e. on densities that may depend on 2nd-order derivative coordinates.
Here, $\Omega=\mathrm{d}z^{1}$ denotes the volume element on $\mathcal{B}$
and $\Omega_{1}=\partial_{1}\rfloor\mathrm{d}z^{1}$ the boundary-volume
form. The integrated quantity of $\mathfrak{H}$, which is given by
$\mathscr{H}=\int_{\mathcal{B}}\mathcal{H}\Omega$, is called the
Hamiltonian functional. The bundle structure $\pi:\mathcal{E}\rightarrow\mathcal{B}$
allows to construct some further geometric objects like the tensor
bundle $\mathcal{W}_{1}^{r}(\mathcal{E})=\mathcal{T}^{*}(\mathcal{E})\wedge\mathcal{T}^{*}(\mathcal{B})$
with a typical element $\omega_{\alpha}\mathrm{d}x^{\alpha}\wedge\mathrm{d}z^{1}$
for $\mathcal{W}_{1}^{r}(\mathcal{E})$, where $\omega_{\alpha}\in C^{\infty}(\mathcal{J}^{r}(\mathcal{E}))$
is met.

\section{Infinite-Dimensional PH-Systems}

This section deals with the pH-system representation based on jet-bundle
structures for systems with 1-dimensional spatial domain and 2nd-order
Hamiltonian, see, e.g. \cite{Rams2017a}. The framework has its origin
in \cite{Ennsbrunner2005,Schoeberl2008a}, and is mainly based on
a certain power-balance relation, which allows us to introduce (power)
ports distributed over the domain as well as on the boundary. In this
paper, we focus on systems with in-domain actuation and hence, as
a classical example we derive a proper pH-system representation of
a piezo-actuated Euler-Bernoulli beam.

Let $\mathfrak{H}$ be a 2nd-order Hamiltonian, i.e. $\mathcal{H}\in C^{\infty}(\mathcal{J}^{2}(\mathcal{E}))$,
then a pH-system formulation including in- and outputs on the domain
is given by
\begin{align}
\dot{x} & =(\mathcal{J}-\mathcal{R})(\delta\mathfrak{H})+u\rfloor\mathcal{G},\label{eq:pH_system_general}\\
y & =\mathcal{G}^{*}\rfloor\delta\mathfrak{H},\nonumber 
\end{align}
together with appropriate boundary conditions.\begin{rem}It should
be noted that in \cite{Schoeberl2011,Schoeberl2013a,Rams2017a}, only
systems with boundary in- and outputs have been investigated. In this
contribution, we focus our interests on systems with in-domain actuation,
and, therefore we include the term $u\rfloor\mathcal{G}$ in our system
representation.\end{rem}In (\ref{eq:pH_system_general}), the interconnection
map $\mathcal{J}$, which describes the internal power flow, and the
dissipation map $\mathcal{R}$ take the form of $\mathcal{J},\mathcal{R}:\mathcal{T}^{*}(\mathcal{E})\wedge\mathcal{T}^{*}(\mathcal{B})\rightarrow\mathcal{V}(\mathcal{E})$.
Furthermore, $\mathcal{J}$ is skew-symmetric, i.e. the coefficients
meet $\mathcal{J}^{\alpha\beta}=-\mathcal{J}^{\beta\alpha}\in C^{\infty}(\mathcal{J}^{4}(\mathcal{E}))$,
and $\mathcal{R}$ is symmetric and positive semidefinite, implying
$\mathcal{R}^{\alpha\beta}=\mathcal{R}^{\beta\alpha}\in C^{\infty}(\mathcal{J}^{4}(\mathcal{E}))$
and $\left[\mathcal{R}^{\alpha\beta}\right]\geq0$ for the coefficient
matrix. For 2nd-order Hamiltonian densities, the variational derivative
corresponds to $\delta\mathfrak{H}=\delta_{\alpha}\mathcal{H}\mathrm{d}x^{\alpha}\wedge\Omega$
with $\delta_{\alpha}(\cdot)=\partial_{\alpha}(\cdot)-d_{1}(\partial_{\alpha}^{1}(\cdot))+d_{11}(\partial_{\alpha}^{11}(\cdot))$.
Due to the fact that we intend to develop in-domain control strategies
in this paper, the terms including the external inputs and collocated
outputs  in (\ref{eq:pH_system_general}) are of particular interest.
It is worth stressing that both, the coefficients $\mathcal{G}_{\xi}^{\alpha}$
of the input map $\mathcal{G}:\mathcal{U}\rightarrow\mathcal{V}(\mathcal{X})$
as well as the input coordinates $u^{\xi}\in\mathcal{U}$, may depend
(amongst others) on the spatial variable $z^{1}$. Based on the duality
of the input- and the output-bundle, see \cite[Section IV]{Ennsbrunner2005},
we are able to deduce the important relation 
\begin{equation}
(u\rfloor\mathcal{G})\delta\mathfrak{H}=u\rfloor(\mathcal{G}^{*}\rfloor\delta\mathfrak{H})=u\rfloor y.\label{eq:distributed_collocation}
\end{equation}
With regard to the control-engineering purposes of the following section,
it is of particular interest how the Hamiltonian functional $\mathscr{H}$
evolves along solutions of the system (\ref{eq:pH_system_general})
(well-posedness provided). If $\mathscr{H}$ corresponds to the total
energy of the system, the formal change, which can be given as
\begin{equation}
\dot{\mathscr{H}}=-\int_{\mathcal{B}}\mathcal{R}(\delta\mathfrak{H})\rfloor\delta\mathfrak{H}+\int_{\mathcal{B}}u\rfloor y+(\dot{x}\rfloor\delta^{\partial,1}\mathfrak{H}+\dot{x}_{1}\rfloor\delta^{\partial,2}\mathfrak{H})|_{\partial\mathcal{B}}\label{eq:h_p_general}
\end{equation}
by means of (\ref{eq:distributed_collocation}), states a power-balance
relation and comprises dissipation and collocation on the domain.
The collocation term $\int_{\mathcal{B}}u\rfloor y$ can be used to
define power ports distributed over the spatial domain allowing for
a non-zero power flow. Furthermore, (\ref{eq:h_p_general}) enables
us to introduce boundary-power ports, where we basically exploit the
boundary operators $\delta^{\partial,1}\mathfrak{H}=(\partial_{\alpha}^{1}\mathcal{H}-d_{1}(\partial_{\alpha}^{11}\mathcal{H}))\mathrm{d}x^{\alpha}\wedge\Omega_{1}$
and $\delta^{\partial,2}\mathfrak{H}=\partial_{\alpha}^{11}\mathcal{H}\mathrm{d}x_{1}^{\alpha}\wedge\Omega_{1}$,
locally given as
\begin{equation}
\delta_{\alpha}^{\partial,1}\mathcal{H}=\partial_{\alpha}^{1}\mathcal{H}-d_{1}(\partial_{\alpha}^{11}\mathcal{H}),\quad\delta_{\alpha}^{\partial,2}\mathcal{H}=\partial_{\alpha}^{11}\mathcal{H}.\label{eq:boundary_operators}
\end{equation}
Worth stressing is the fact that for the system class under investigation
(1-dimensional spatial domain and 2nd-order Hamiltonian), the formal
change (\ref{eq:h_p_general}) can be determined by integration by
parts. However, for pH-systems with 2nd-order Hamiltonian and higher-dimensional
spatial domain, the calculation of $\dot{\mathscr{H}}$ is a non-trivial
task, which is treated in \cite{Schoeberl2018}.

A local system representation of (\ref{eq:pH_system_general}) can
be given as
\begin{align}
\dot{x}^{\alpha} & =(\mathcal{J}^{\alpha\beta}-\mathcal{R}^{\alpha\beta})\delta_{\beta}\mathcal{H}+\mathcal{G}_{\xi}^{\alpha}u^{\xi},\label{eq:pH_non_differential_operator}\\
y_{\xi} & =\mathcal{G}_{\xi}^{\alpha}\delta_{\alpha}\mathcal{H},\nonumber 
\end{align}
with $\alpha,\beta=1,\ldots,n$ and $\xi=1,\ldots,m$. Henceforth,
as we focus on systems actuated solely within the spatial domain,
we suppose that no power exchange takes place through the boundary
$\partial\mathcal{B}=\{0,L\}$, i.e. $(\dot{x}^{\alpha}\delta_{\alpha}^{\partial,1}\mathcal{H})|_{\partial\mathcal{B}}=0$
as well as $(\dot{x}_{1}^{\alpha}\delta_{\alpha}^{\partial,2}\mathcal{H})|_{\partial\mathcal{B}}=0$.
Consequently, the power-balance relation (\ref{eq:h_p_general}) follows
to
\begin{equation}
\dot{\mathscr{H}}=-\int_{\mathcal{B}}\delta_{\alpha}(\mathcal{H})\mathcal{R}^{\alpha\beta}\delta_{\beta}(\mathcal{H})\mathrm{d}z^{1}+\int_{\mathcal{B}}u^{\xi}y_{\xi}\mathrm{d}z^{1}\label{eq:h_p_general_local}
\end{equation}
in local coordinates.

As an example, an Euler-Bernoulli beam actuated by one pair of piezoelectric
macro-fibre composite (MFC) patches is studied. To this end, we summarise
the derivation of the equation of motion for the transversal deflection
$w$ of the beam, which is given in detail in \cite{Schroeck2011}.
Furthermore, we aim to find a pH-system representation being suitable
for the control strategy presented in Section 4.

\begin{exmp}[Piezo-actuated Euler-Bernoulli beam]\label{ex:Euler-Bernoulli_Beam_Piezo_actuated}We
consider an Euler-Bernoulli beam actuated by one pair of piezoelectric
patches with two symmetrically placed actuators on the upper and lower
side of the beam. Furthermore, we assume that the beam is clamped
at the position $z^{1}=0$, i.e. $w(t,0)=0$, $\dot{w}(t,0)=0$ and
$w_{1}(t,0)=0$, while the other end $z^{1}=L$ is free, implying
that the shear force and the bending moment vanish.

First, we derive the equation of motion in a Lagrangian framework
by exploiting the calculus of variations, see \cite{Meirovitch1997}
for instance. In this setting, the time $t$ and the spatial coordinate
$z^{1}$ are used as independent variables. To begin with, we state
the energy densities of the system under investigation, where we first
focus on the part of the energy which is due to the beam. If we use
linear constitutive and linearised geometric relations according to
the Euler-Bernoulli hypothesis, the potential-energy density of the
beam can be given as
\[
\mathcal{V}_{b}=\tfrac{1}{2}EI(w_{11})^{2},
\]
with $E$ and $I$ denoting Young's modulus and the moment of inertia,
respectively. Moreover, the kinetic-energy density of the beam reads
as
\[
\mathcal{T}_{b}=\tfrac{1}{2}\rho_{b}A_{b}(w_{t})^{2},
\]
where $\rho_{b}$ is the mass density and $A_{b}$ the cross section
of the beam. With regard to the MFC patch pair, it is important to
mention that it is attached at a specified position. To mathematically
describe the position of the piezoelectric pair, we introduce the
spatial actuator characteristic
\begin{equation}
\Gamma(z^{1})=h(z^{1}-z_{p}^{1})-h(z^{1}-z_{p}^{1}-L_{p}),\label{eq:piezo_spatial_function}
\end{equation}
with $h(\cdot)$ denoting the Heaviside function and $z_{p}^{1}$
the position where the MFC patches of the length $L_{p}$ are attached
meeting $0<z_{p}^{1}<z_{p}^{1}+L_{p}<L$. Consequently, the kinetic-energy
density of the MFC patches follows to
\[
\mathcal{T}_{p}=\rho_{p}A_{p}\Gamma(z^{1})(w_{t})^{2},
\]
where the corresponding cross section and mass density are denoted
by $A_{p}$ and $\rho_{p}$, respectively. To keep the complexity
as low as possible, we assume the following simplifications regarding
the MFC patches. First, we suppose a perfect compensation of all actuator
nonlinearities as well as an uniaxial state of stress. Furthermore,
we describe the electric field between the electrodes by an exclusive
field component $E_{1}$, i.e. $E_{2}=E_{3}=0$, and additionally
neglect the self-generated electric field stemming from the direct
piezoelectric effect as it is irrelevant compared to $E_{1}$. Consequently,
if we use linear constitutive relations for the MFC patches, see \cite[equ. (7)]{Schroeck2011},
and take the preceding assumptions into account, the potential-energy
density follows to
\[
\mathcal{V}_{p}=\Gamma(z^{1})(\Theta_{p}(w_{11})^{2}+2\Delta_{p}w_{11}u_{in}),
\]
where the (constant) material parameters of the MFC patches are hidden
in the abbreviations $\Theta_{p}$ and $\Delta_{p}$. Furthermore,
$u_{in}$ denotes the input voltage of the piezoelectric actuators.
To derive the equation of motion together with the boundary conditions,
we use the 2nd-order Lagrangian density
\begin{equation}
\mathcal{L}=\mathcal{T}_{b}+\mathcal{T}_{p}-\mathcal{V}_{b}-\mathcal{V}_{p},\label{eq:Langrangian_density}
\end{equation}
which can be given as the difference of the total-kinetic energy
density
\begin{equation}
\mathcal{T}=\tfrac{1}{2}\kappa(z^{1})(w_{t})^{2}\label{eq:total_kinetic_energy_density}
\end{equation}
and the total-potential energy density
\begin{equation}
\mathcal{V}=\tfrac{1}{2}\Theta(z^{1})(w_{11})^{2}+2\Gamma(z^{1})\Delta_{p}w_{11}u_{in}\label{eq:total_potential_energy_density}
\end{equation}
by means of the spatially varying parameters $\kappa(z^{1})=\rho_{b}A_{b}+2\rho_{p}A_{p}\Gamma(z^{1})$
and $\Theta(z^{1})=EI+2\Theta_{p}\Gamma(z^{1})$. For the system under
consideration -- 1st-order derivative variable with respect to $t$
and (solely) 2nd-order derivative variables with respect to $z^{1}$
--, the Euler-Lagrange operator corresponds to
\begin{equation}
\delta_{w}(\cdot)\!=\!\partial_{w}(\cdot)-d_{t}(\partial_{w}^{t}(\cdot))+d_{11}(\partial_{w}^{11}(\cdot)),\label{eq:Euler_Lagrange_Domain}
\end{equation}
and the boundary operators are\begin{subequations}\label{eq:Euler_Lagrange_Boundary}
\begin{align}
\delta_{w}^{\partial,1}(\cdot) & =\partial_{w}^{1}(\cdot)-d_{1}(\partial_{w}^{11}(\cdot)),\\
\delta_{w}^{\partial,2}(\cdot) & =\partial_{w}^{11}(\cdot),
\end{align}
\end{subequations}as in mechanics it is common to allow for no variation
on the time boundary. If we apply the domain operator (\ref{eq:Euler_Lagrange_Domain})
and the boundary operators (\ref{eq:Euler_Lagrange_Boundary}) to
the Lagrangian density (\ref{eq:Langrangian_density}), due to the
requirements $\delta_{w}\mathcal{L}=0$ as well as $(\dot{w}\delta_{w}^{\partial,1}\mathcal{L})|_{\partial\mathcal{B}}=0$,
$(\dot{w}_{1}\delta_{w}^{\partial,2}\mathcal{L})|_{\partial\mathcal{B}}=0$,
we obtain the equation of motion
\begin{equation}
\kappa(z^{1})w_{tt}=-\Theta(z^{1})w_{1111}-4\Theta_{p}\partial_{1}(\Gamma(z^{1}))w_{111}-2\Theta_{p}\partial_{11}(\Gamma(z^{1}))w_{11}-2\Delta_{p}\partial_{11}(\Gamma(z^{1}))u_{in},\label{eq:PDE_beam_piezo}
\end{equation}
together with the boundary conditions\begin{subequations}\label{eq:boundary_cond_Piezo}
\begin{align}
(\dot{w}EIw_{111})|_{\partial\mathcal{B}} & =0,\\
(\dot{w}_{1}EIw_{11})|_{\partial\mathcal{B}} & =0.
\end{align}
\end{subequations}It should be noted that the spatial derivatives
of $\Gamma(z^{1})$ occuring in (\ref{eq:PDE_beam_piezo}) would cause
some problems regarding the formulation of the equation of motion.
To avoid this problem, we approximate the discontinuous characteristic
(\ref{eq:piezo_spatial_function}) by the spatially differentiable
function
\[
\Gamma(z^{1})=\tfrac{1}{2}\tanh(\sigma(z^{1}-z_{p}^{1}))-\tfrac{1}{2}\tanh(\sigma(z^{1}-z_{p}^{1}-L_{p}))
\]
with the scaling factor $\sigma\in\mathbb{R}_{+}$.

Next, we are interested in a proper pH-system representation for the
system under investigation. It should be noted that in the pH-setting
the time $t$ plays the role of an evolution parameter, i.e. $t$
is no coordinate any more and, hence, the exclusive independent variable
is the spatial coordinate $z^{1}$. Consequently, we consider the
bundle $\pi:\mathcal{E}\rightarrow\mathcal{B}$ with the independent
coordinate $(z^{1})$ for $\mathcal{B}$ and coordinates $(z^{1},w,p)$
for $\mathcal{E}$ in the following, where we have introduced the
generalised momenta $p=\kappa(z^{1})\dot{w}=\kappa(z^{1})w_{t}$.
In principle, the Hamiltonian density is chosen as the sum of potential-
and kinetic-energy density according to $\mathcal{H}=\mathcal{T}+\mathcal{V}$.
However, to obtain an appropriate pH-system formulation, we set 
\[
\mathcal{H}=\tfrac{1}{2\kappa(z^{1})}p^{2}+\tfrac{1}{2}\Theta(z^{1})(w_{11})^{2},
\]
where we intentionally omit the term $2\Delta_{p}\Gamma(z^{1})w_{11}u_{in}$
of (\ref{eq:total_potential_energy_density}), as by using the calculus
of variations it generates the input part in (\ref{eq:PDE_beam_piezo}),
which is hidden in $g(z^{1})=-2\Delta_{p}\partial_{11}(\Gamma(z^{1}))$
in the pH-system representation

\begin{align}
\left[\begin{array}{c}
\dot{w}\\
\dot{p}
\end{array}\right] & =\left[\begin{array}{cc}
0 & 1\\
-1 & 0
\end{array}\right]\left[\begin{array}{c}
\delta_{w}\mathcal{H}\\
\delta_{p}\mathcal{H}
\end{array}\right]+\left[\begin{array}{c}
0\\
g(z^{1})
\end{array}\right]u_{in},\nonumber \\
y & =\left[\begin{array}{cc}
0 & g(z^{1})\end{array}\right]\left[\begin{array}{c}
\delta_{w}\mathcal{H}\\
\delta_{p}\mathcal{H}
\end{array}\right]=g(z^{1})\dot{w}.\label{eq:Piezo_pH}
\end{align}
Moreover, we are able to deduce the formal change of the Hamiltonian
functional, which follows to
\begin{equation}
\dot{\mathscr{H}}=\int_{\mathcal{B}}g(z^{1})\dot{w}u_{in}\mathrm{d}z^{1}\label{eq:h_p_piezo_beam}
\end{equation}
as the boundary terms vanish due to the boundary conditions (\ref{eq:boundary_cond_Piezo}).
It is worth stressing that (\ref{eq:h_p_piezo_beam}) corresponds
to an electrical power-balance relation as the unit of the distributed
output density (\ref{eq:Piezo_pH}) is $\mathrm{\tfrac{A}{m}}$.\end{exmp}
Ex. (\ref{ex:Euler-Bernoulli_Beam_Piezo_actuated}) highlights that
external inputs together with the collocated outputs generate power
ports which may be distributed over (a part of) the spatial domain.
As distributed ports allow for a non-zero power flow over the domain,
we use them to couple an infinite-dimensional system to a dynamic
pH-controller in the following section. 

\section{In-Domain Control by means of Structural Invariants}

This section deals with the extension of the energy-Casimir method
to infinite-dimensional pH-systems with in-domain actuation. Here,
we confine ourselves to systems with lumped inputs that may act distributed
over a part of the spatial domain. This has the consequence that the
collocated outputs can be interpreted as distributed output densities.
However, a certain interconnection allows for the use of a finite-dimensional
dynamic controller. An advantage of the proposed control strategy
is the applicability to piezo-actuated beams, which is presented at
the end of this section. 

\subsection{Interconnection (Infinite-Finite)}

In the following, we aim at stabilising pH-systems of the form (\ref{eq:pH_non_differential_operator}),
where the lumped inputs $u^{\xi}$ may act distributed over a part
of the spatial domain due to the input-map components $\mathcal{G}_{\xi}^{\alpha}$,
cf. Ex. \ref{ex:Euler-Bernoulli_Beam_Piezo_actuated}. To this end,
we are interested in a power-conserving interconnection of the infinite-dimensional
plant (\ref{eq:pH_non_differential_operator}) and a finite-dimensional
pH-controller, given in local coordinates as
\begin{align}
\dot{x}_{c}^{\alpha_{c}} & =(J_{c}^{\alpha_{c}\beta_{c}}-R_{c}^{\alpha_{c}\beta_{c}})\partial_{\beta_{c}}H_{c}+G_{c,\xi}^{\alpha_{c}}u_{c}^{\xi},\nonumber \\
y_{c,\xi} & =G_{c,\xi}^{\alpha_{c}}\partial_{\alpha_{c}}H_{c},\label{eq:pH_controller_finite}
\end{align}
with $\alpha_{c},\beta_{c}=1,\ldots,n_{c}$ and $\xi=1,\ldots,m$.
There, one must take account of the fact that the outputs of (\ref{eq:pH_non_differential_operator})
are considered as distributed output densities. Thus, to enable a
coupling with the finite-dimensional controller, the output densities
must be integrated over $\mathcal{B}$ and therefore, we choose a
power-conserving interconnection of the form
\begin{equation}
u^{\xi}\int_{\mathcal{B}}y_{\xi}\mathrm{d}z^{1}+u_{c}^{\xi}y_{c,\xi}=0.\label{eq:PCIS_3}
\end{equation}
To obtain a power-conserving interconnection which meets (\ref{eq:PCIS_3}),
we couple the infinite-dimensional plant and the finite-dimensional
controller according to 
\begin{equation}
u_{c}^{\xi}=K^{\xi\eta}\int_{\mathcal{B}}y_{\eta}\mathrm{d}z^{1},\quad u^{\xi}=-K^{\xi\eta}y_{c,\eta},\label{eq:PCIS_3_coupling}
\end{equation}
with $K^{\xi\eta}$ denoting the components of an appropriate map
$K$. It is worth stressing that the closed-loop system, which is
a result of the interconnection (\ref{eq:PCIS_3_coupling}), still
possesses a pH-structure, with the closed-loop Hamiltonian $\mathscr{H}_{cl}=\int_{\mathcal{B}}\mathcal{H}\mathrm{d}z^{1}+H_{c}$.
As we consider systems where no power exchange takes place at the
boundary $\partial\mathcal{B}$, i.e. $(\dot{x}^{\alpha}\delta_{\alpha}^{\partial,1}\mathcal{H}+\dot{x}_{1}^{\alpha}\delta_{\alpha}^{\partial,2}\mathcal{H})|_{\partial\mathcal{B}}=0$
is valid, the formal change of the closed-loop Hamiltonian $\mathscr{H}_{cl}$
can be deduced to
\[
\dot{\mathscr{H}}_{cl}=-\int_{\mathcal{B}}\delta_{\alpha}(\mathcal{H})\mathcal{R}^{\alpha\beta}\delta_{\beta}(\mathcal{H})\mathrm{d}z^{1}-\partial_{\alpha_{c}}(H_{c})R_{c}^{\alpha_{c}\beta_{c}}\partial_{\beta_{c}}(H_{c})
\]
because of the coupling (\ref{eq:PCIS_3_coupling}).

Having defined the coupling of the plant (\ref{eq:pH_non_differential_operator})
and the controller (\ref{eq:pH_controller_finite}) by means of (\ref{eq:PCIS_3_coupling}),
we are interested in structural invariants of the closed-loop system
aiming at relating some of the controller states to the plant. By
means of these controller states, we partially shape $\mathscr{H}_{cl}$,
whereas the controller states that are not related to the plant shall
be used for the damping injection and thus for the purpose of stabilisation.\begin{rem}\label{rem:stability}At
this point, it should be mentioned again that stability investigations
for systems governed by PDEs usually require functional-analytic methods.
In this contribution, the focus is on a formal approach exploiting
geometric system properties and consequently, no detailed stability
investigations will be carried out. However, worth stressing is the
fact that $\mathscr{H}_{cl}>0$ and $\dot{\mathscr{H}}_{cl}\leq0$
can be used for stability investigations in the sense of Lyapunov.\end{rem}

\subsection{Structural Invariants}

Motivated by the form of the plant (distributed) and the controller
(lumped), we consider Casimir-functionals according to
\begin{equation}
\mathscr{C}^{\lambda}=x_{c}^{\lambda}+\int_{\mathcal{B}}\mathcal{C}^{\lambda}\mathrm{d}z^{1},\label{eq:Casimir_Functionals_Piezo}
\end{equation}
with $\lambda=1,\ldots,\bar{n}\leq n_{c}$. The functionals (\ref{eq:Casimir_Functionals_Piezo})
must be constant along solutions of the closed loop, i.e. $\dot{\mathscr{C}}^{\lambda}=0$
is valid independently of $\mathcal{H}$ and $H_{c}$, in order to
serve as structural invariants.

\begin{prop}\label{prop:Casimir_Conditions_Piezo}Consider the closed-loop
system which stems from the coupling of (\ref{eq:pH_non_differential_operator})
and (\ref{eq:pH_controller_finite}) via the interconnection (\ref{eq:PCIS_3_coupling}).
Then, the functionals (\ref{eq:Casimir_Functionals_Piezo}) are structural
invariants of the closed loop, iff the conditions\begin{subequations}\label{eq:Casimir_Conditions_Piezo}
\begin{align}
(J_{c}^{\lambda\beta_{c}}-R_{c}^{\lambda\beta_{c}}) & =0\\
\delta_{\alpha}\mathcal{C}^{\lambda}(\mathcal{J}^{\alpha\beta}-\mathcal{R}^{\alpha\beta})+G_{c,\xi}^{\lambda}K^{\xi\eta}\mathcal{G}_{\eta}^{\beta} & =0\label{eq:domain_cond_prop_Piezo}\\
\delta_{\alpha}\mathcal{C}^{\lambda}\mathcal{G}_{\xi}^{\alpha}K^{\xi\eta}G_{c,\eta}^{\alpha_{c}} & =0\label{eq:Casimir_Input_Obstacle}\\
(\dot{x}^{\alpha}\delta_{\alpha}^{\partial,1}\mathcal{C}^{\lambda}+\dot{x}_{1}^{\alpha}\delta_{\alpha}^{\partial,2}\mathcal{C}^{\lambda})|_{\partial\mathcal{B}} & =0\label{eq:Casimir_Cond_Piezo_Boundary}
\end{align}
\end{subequations}are fulfilled.\end{prop}

\begin{proof}We begin by calculating the formal change of (\ref{eq:Casimir_Functionals_Piezo})
along trajectories of the closed loop, which follows to
\begin{equation}
\dot{\mathscr{C}}^{\lambda}=\dot{x}_{c}^{\lambda}+\int_{\mathcal{B}}\dot{x}^{\alpha}\delta_{\alpha}\mathcal{C}^{\text{\ensuremath{\lambda}}}\mathrm{d}z^{1}+(\dot{x}^{\alpha}\delta_{\alpha}^{\partial,1}\mathcal{C}^{\lambda}+\dot{x}_{1}^{\alpha}\delta_{\alpha}^{\partial,2}\mathcal{C}^{\lambda})|_{\partial\mathcal{B}},\label{eq:Casimir_Conditions_Piezo_C_p}
\end{equation}
and take into account the requirement $\dot{\mathscr{C}}^{\lambda}=0$
independently of $\mathcal{H}$ and $H_{c}$. If we substitute the
system equations of the controller and the plant (\ref{eq:pH_controller_finite})
and (\ref{eq:pH_non_differential_operator}), respectively, as well
as the relations (\ref{eq:PCIS_3_coupling}), then we obtain
\begin{multline*}
\dot{\mathscr{C}}^{\lambda}=(J_{c}^{\lambda\beta_{c}}-R_{c}^{\lambda\beta_{c}})\partial_{\beta_{c}}H_{c}+\int_{\mathcal{B}}(\delta_{\alpha}\mathcal{C}^{\lambda}(\mathcal{J}^{\alpha\beta}-\mathcal{R}^{\alpha\beta})+G_{c,\xi}^{\lambda}K^{\xi\eta}\mathcal{G}_{\eta}^{\beta})\delta_{\beta}\mathcal{H}\mathrm{d}z^{1}+\ldots\\
-\int_{\mathcal{B}}\delta_{\alpha}\mathcal{C}^{\lambda}\mathcal{G}_{\xi}^{\alpha}K^{\xi\eta}G_{c,\eta}^{\alpha_{c}}\partial_{\alpha_{c}}H_{c}\mathrm{d}z^{1}+(\dot{x}^{\alpha}\delta_{\alpha}^{\partial,1}\mathcal{C}^{\lambda}+\dot{x}_{1}^{\alpha}\delta_{\alpha}^{\partial,2}\mathcal{C}^{\lambda})|_{\partial\mathcal{B}}
\end{multline*}
and the proof follows immediately.\end{proof} Now, it is interesting
to interpret the results of Prop. \ref{prop:Casimir_Conditions_Piezo}
and draw some conclusions to the results and findings of \cite[Prop. 1]{Rams2017a}
for the boundary control of 1-dimensional pH-systems with 2nd-order
Hamiltonian. In particular, the condition (\ref{eq:domain_cond_prop_Piezo})
is of special interest as it allows to relate the plant within the
domain to $\lambda=1,\ldots,\bar{n}\leq n_{c}$ controller states,
which is not possible with the Casimir condition \cite[Eq. (17b)]{Rams2017a}.
Unfortunately, we are not able to relate every system state to the
plant, since the condition (\ref{eq:Casimir_Input_Obstacle}) describes
the fact that we cannot find any Casimir functions depending on variables
where the inputs of the plant appear in the corresponding system equations.
In \cite{Rams2017a}, the boundary of the infinite-dimensional system
is divided into an actuated and an unactuated part, where the actuated
boundary is used to relate the plant to the controller, see \cite[Eqs. (17c) and (17d)]{Rams2017a}.
In contrast, as we solely consider a coupling of the plant within
the spatial domain, the impact of the Casimir functionals at the boundary
must vanish, cf. (\ref{eq:Casimir_Cond_Piezo_Boundary}), like in
\cite[Eq. (17e)]{Rams2017a} at the unactuated boundary.

Having the preceding findings at hand, the piezo-actuated beam of
Ex. \ref{ex:Euler-Bernoulli_Beam_Piezo_actuated} serves us to demonstrate
the applicability of the proposed control scheme.

\begin{exmp}[Energy-Casimir controller for Ex. 2]\label{ex:Controler_Beam_Piezo}In
this example, we develop a finite-dimensional pH-controller for the
piezo-actuated Euler-Bernoulli beam of Ex. \ref{ex:Euler-Bernoulli_Beam_Piezo_actuated},
aiming at stabilising the equilibrium
\begin{equation}
w^{d}\!=\!\left\{ \begin{array}{cc}
0 & \mathrm{for}\:0\leq z^{1}<z_{p}^{1},\\
a(z^{1}-z_{p}^{1})^{2} & \mathrm{for}\:z_{p}^{1}\leq z^{1}<z_{p}^{1}+L_{p}\\
b(z^{1}-z_{p}^{1}-L_{p})+a(L_{p})^{2} & \mathrm{for}\:z_{p}^{1}+L_{p}\leq z_{p}^{1}\leq L.
\end{array},\right.\label{eq:equilibrium_piezo_beam}
\end{equation}
To this end, we relate one controller state to the plant, and two
controller states shall be used for the damping injection, i.e., the
dimension of the controller follows to $n_{c}=3$. By means of a proper
choice of the initial controller states, the Casimir function $\mathcal{C}^{1}=-g(z^{1})w$,
which satisfies the conditions (\ref{eq:Casimir_Conditions_Piezo})
if we set $G_{c}^{1}=1$ and $K=1$, yields the important relation
\[
x_{c}^{1}=\int_{\mathcal{B}}g(z^{1})w\mathrm{d}z^{1}.
\]
 It should be noted that $x_{c}^{1}$ corresponds to the weighted,
integrated deflection of the beam, whereas in \cite{Rams2017a} the
deflection and angle at the actuated boundary are used as controller
states. Furthermore, the controller dynamics are constrained to the
maps
\[
J_{c}-R_{c}\!=\!\left[\begin{array}{ccc}
0 & 0 & 0\\
0 & -R_{c}^{22} & J_{c}^{23}-R_{c}^{23}\\
0 & -J_{c}^{23}-R_{c}^{23} & -R_{c}^{33}
\end{array}\right],\:G_{c}\!=\!\left[\begin{array}{c}
1\\
G_{c}^{2}\\
G_{c}^{3}
\end{array}\right]
\]
due to the conditions (\ref{eq:Casimir_Conditions_Piezo}). To be
able to stabilise the equilibrium (\ref{eq:equilibrium_piezo_beam}),
it must become a part of the minimum of the closed-loop Hamiltonian
$\mathscr{H}_{cl}=\mathscr{H}+H_{c}$. For this purpose, we set the
controller Hamiltonian to
\[
H_{c}=\tfrac{c_{1}}{2}(x_{c}^{1}-x_{c}^{1,d}-\tfrac{u_{s}}{c_{1}})^{2}+\tfrac{1}{2}M_{c,\mu_{c}\nu_{c}}x_{c}^{\mu_{c}}x_{c}^{\nu_{c}},
\]
together with the relation $x_{c}^{1,d}=\int_{\mathcal{B}}g(z^{1})w^{d}\mathrm{d}z^{1}$,
the positive definite matrix $[M_{c}]$, $M_{c,\mu_{c}\nu_{c}}\in\mathbb{R}$
for $\mu_{c},\nu_{c}=2,3$, and the positive constant $c_{1}>0$.
Note that the term with $u_{s}$ was incorporated in $H_{c}$ as (\ref{eq:equilibrium_piezo_beam})
is an equilibrium that requires non-zero power, i.e. a stationary
voltage $u_{s}$ inducing a static holding torque. If we consider
the power-conserving interconnection structure $u_{c}=\int_{\mathcal{B}}g(z^{1})\dot{w}\mathrm{d}z^{1}$
and $u=-y_{c}$, the formal change of $\mathscr{H}_{cl}$ follows
to
\begin{equation}
\dot{\mathscr{H}}_{cl}=-x_{c}^{\mu_{c}}M_{c,\mu_{c}\nu_{c}}R_{c}^{\nu_{c}\rho_{c}}M_{c,\rho_{c}\vartheta_{c}}x_{c}^{\vartheta_{c}}\leq0,\label{eq:H_cl_p_Piezo}
\end{equation}
with $\rho_{c},\vartheta_{c}=2,3$. Since in this contribution no
detailed stability investigations are carried out, cf. Rem. \ref{rem:stability},
we are content with the fact that $\dot{\mathscr{H}}_{cl}\leq0$ implies
that $\mathscr{H}_{cl}$ is non-increasing along closed-loop solutions.
Furthermore, Fig. \ref{fig:sim_result} shows that the proposed controller
stabilises the desired equilibrium (\ref{eq:equilibrium_piezo_beam})
with $a=0.3587$ and $b=0.1436$. Here, all system parameters of the
beam are set to 1 and the piezoelectric patches with the length $L_{p}=0.2$
are placed at $z_{p}^{1}=0.2$. The remaining degrees of freedom for
the pH-controller are chosen according to $J_{c}^{23}=1$, $R_{c}^{22}=3$,
$R_{c}^{23}=-1$, $R_{c}^{33}=1.5$, $M_{c,22}=85$, $M_{c,23}=0$,
$M_{c,33}=60$, $G_{c}^{2}=G_{c}^{3}=1.7$ and $c_{1}=0.1$.
\begin{figure}
\center\input{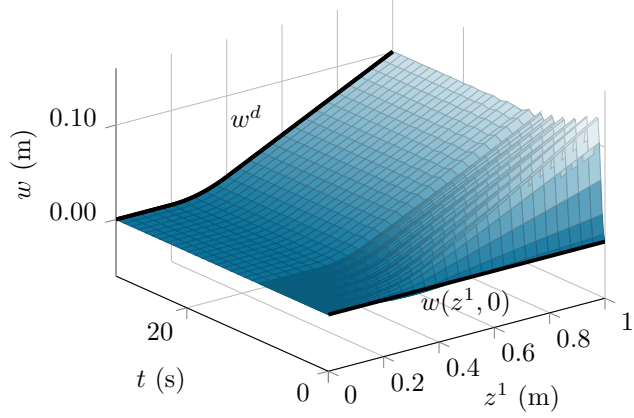}\caption{\label{fig:sim_result}Simulation results for the beam deflection
$w$ against time $t$ and spatial domain $z^{1}$.}
\end{figure}
\end{exmp}

\section{Conclusion and Outlook}

In this article, a control strategy for infinite-dimensional pH-systems
with in-domain actuation yielding a dynamic controller has been developed.
Here, we focused on systems with lumped inputs that act distributed
over a part of the spatial domain and demonstrated the applicability
of the control approach by means of a piezo-actuated Euler-Bernoulli
beam. Since the distributed output density of the system under investigation
cannot be measured, for the implementation of the controller an observer
would be required, which is part of future research. Furthermore,
we aim to extend the proposed control strategy to infinite-dimensional
systems with distributed inputs requiring an infinite-dimensional
controller.

\bibliographystyle{ieeetr}
\bibliography{my_bib}

\end{document}